\documentclass[11pt]{article}

\usepackage{pgf,tikz}
\usepackage{mathrsfs}
\usetikzlibrary{arrows}

\usepackage{graphicx} 
\graphicspath{{figures/}} 

\usepackage{fullpage, url,amsmath,amsfonts,amssymb,mathtools,mathrsfs,graphicx,  algorithm, float, sansmath,epstopdf,color,caption,enumitem,tabularx}
\usepackage[final]{pdfpages}
\usepackage{amsthm}
\usetikzlibrary{automata,topaths}
\usetikzlibrary{decorations.pathreplacing,shapes.misc}
\usepackage{fancyhdr}

\def\beq{\begin{equation}}
\def\eeq{\end{equation}}
\def\baq{\begin{eqnarray}}
\def\eaq{\end{eqnarray}}
\def\baqn{\begin{eqnarray*}}
\def\eaqn{\end{eqnarray*}}

\newcommand{\ball}{\mathbb{B}}

\usepackage{multirow}
\usetikzlibrary{calc,arrows}
\theoremstyle{plain}

\newtheorem{remark}{Remark}

\newtheorem{example}{Example}
\newtheorem{theorem}{Theorem}
\newtheorem{lemma}[theorem]{Lemma}

\usepackage{blindtext}

\usepackage[colorlinks,linkcolor=blue,citecolor=red]{hyperref}

\usepackage{xcolor, framed}

\newcommand{\R}{{\mathbb R}}

\newcommand{\interior}{{\rm int}\kern 0.06em}

\def\<{\langle}
\def\>{\rangle}

\usepackage{ulem}

\usepackage{subcaption}

\usepackage{pict2e}

\if
{

\usepackage[pageref]{backref}
\renewcommand*{\backrefalt}[4]{%
\ifcase #1 %
(Not cited)%
\or
(Cited on p.~#2)%
\else
(Cited on pp.~#2)%
\fi
}

}
\fi

\begin{document}
\title{Sliding Mode Observer for Set-valued Lur'e Systems and Chattering Removing}
\author{Samir Adly\thanks{Laboratoire XLIM, Universit\'e de Limoges,
123 Avenue Albert Thomas,
87060 Limoges CEDEX, France\vskip 0mm
Email: \texttt{samir.adly@unilim.fr}} \qquad Ba Khiet Le \thanks{Optimization Research Group, Faculty of Mathematics and Statistics, Ton Duc Thang University, Ho Chi Minh City, Vietnam\vskip 0mm
 E-mail: \texttt{lebakhiet@tdtu.edu.vn}}
 }
\date{}
\maketitle

\begin{abstract}
{In this paper, we study} a sliding mode observer for a class of set-valued Lur'e systems subject to uncertainties. 
{We show the well-posedness of the problem and highlight the clear advantages of our approach over the existing Luenberger-like observers.} {Furthermore, we provide a new continuous approximation to {remove} the chattering effect in the sliding mode technique.} {Some numerical examples are given to illustrate our theoretical approach.}
\end{abstract}



\section{Introduction}
{Hybrid systems are a class of dynamic systems that may present both continuous and discrete behavior. They are characterized by the presence of continuous state variables, inputs/outputs, as well as discrete state variables, inputs/outputs. On the other hand, Lur'e-type nonlinear systems \cite{Lurie} are represented by the combination of a linear time-invariant system and a memoryless nonlinearity through a feedback connection.
%
%
The link between Lur'e systems and hybrid systems lies in the coupling of the continuous and discrete dynamics through the nonlinear state equation and the linear output equation. This coupling makes Lur'e systems well-suited for modeling complex physical, biological, or social systems that exhibit both continuous and discrete behavior, and have found applications in areas such as control systems, signal processing, and system identification. For a comprehensive guide to hybrid dynamical systems from the modeling, stability and robustness point of view, we refer to \cite{GST}. In this paper we focus on Lur'e set-valued dynamical systems where the nonlinear feedback is given by a set-valued relation.} {The Lur'e set-valued dynamical system is a widely studied model in applied mathematics and control theory. In recent decades, it has received significant attention, as evidenced by the numerous studies found in references such as \cite{Acary,ahl,ahl2,bg,bg2,brogliato,bh,bt,cs,Huang,Huang1,Huang3,L1,L2,L3,abc0}. Despite advancements in the understanding of the system's existence and uniqueness, stability and asymptotic analysis, the design of control and observers still poses many interesting open questions. Specifically, most observers for Lur'e set-valued systems follow the Luenberger design, which has limitations when the system is subjected to uncertainty, as discussed in references such as \cite{Huang,Huang1,Huang3,abc0}.}
Recently, using the powerful sliding mode technique, B. K. Le in \cite{L3} proposed a 
 sliding mode observer for a general class of set-valued Lur'e systems subject to uncertainties as follows
\beq
\left\{
\begin{array}{l}
   \dot{x}(t) = Ax(t)+B\lambda(t)+Eu(t)+G\xi(t,u,x)\; {\rm for\; a.e.} \; t \in [0,+\infty), \label{1a}\\ \\
   w(t)=Cx(t)+D\lambda(t),\\ \\
   \lambda(t) \in -\mathcal{F}_t(w(t)), \;t\ge 0,\\ \\
   y(t)=Fx(t), \\ \\
   x(0) = x_0,
\end{array}\right.
\eeq
where      $ A\in \R^{n\times n}, B \in \R^{n\times m},C\in \R^{m\times n}, D\in \R^{m\times m}, E\in \R^{n\times l}, F\in \R^{p\times n}, G\in \R^{n\times k}$ are given matrices, $\lambda$ and $w $ are two  connecting  variables, $u$ is the control input, $y$ is the physically measurable output and $\xi$ is some uncertainty.  
{The set-valued operator $\mathcal{F}_t:\R^m\rightrightarrows\R^m$ is time-dependent and assumed to be maximal monotone. }
Note that  $(1)$ can be rewritten into a first order time-dependent differential inclusion as follows
\beq\label{equidi}
 \dot{x} \in  Ax-B(\mathcal{F}_t^{-1}+D)^{-1}Cx+Eu+G\xi, \;\;x(0)=x_0.
 \eeq
If $D=0$ and $\mathcal{F}_t=\mathcal{F}$, we obtain the {following classical case}
 $$
  \dot{x} \in  Ax-B\mathcal{F}Cx+Eu+G\xi.
 $$
The proposed sliding mode observer for $(1)$ is 

\beq
\left\{
\begin{array}{l}
   \dot{\tilde{x}} \in  A\tilde{x}+B\tilde{\lambda}-Le_y+Eu-P^{-1}F^T\Psi(e_y),\\ \\
   \tilde{w}=C\tilde{x}+D\tilde{\lambda},\\ \\
   \tilde{\lambda} \in -\mathcal{F}(\tilde{w}+Ke_y),\\ \\
   \tilde{y}=F\tilde{x},
\end{array}\right.
\eeq
where
$$
e=\tilde{x}-x, e_y=\tilde{y}-y=Fe,
$$
\baq\label{setv}\nonumber
\Psi(e_y)=\sigma_1 e_y+ (\Vert J \Vert\ \rho(t,u,y)
+\sigma_2){\rm Sign}(e_y)
\eaq
and $\rho$ is a bound of the uncertainty for some suitable matrix $J$ and real number $\sigma_1\ge 0, \sigma_2>0$. Then under some mild conditions, the observer state converges to the original state asymptotically and the observation error converges to zero in finite time ({we refer to \cite{L3} for more details}). 

   In this paper, our first contribution is to provide a sliding mode observer for the system considered in \cite{Huang} as follows 
\beq \left\{
\begin{array}{lll}
\; \dot{x}=Ax +B\omega +f_1(x,u)+f_2(x,u)\theta(t),\\
&&\\
\;\omega \in -\mathcal{F}(Cx), \\
&&\\
\; y=Fx,
\end{array}\right. \label{sysh}
\eeq
where $A\in \R^{n\times n}, B\in \R^{n\times m}, C\in \R^{m\times n}, F\in \R^{p\times n}$ are given matrix, $x\in \R^n$ is the state, $\mathcal{F}: \R^m\rightrightarrows  \R^m$ is a maximal monotone operator,  $u\in \R^r$ is the control input and $y\in \R^p$ is the measurable output. The functions $f_1,f_2$ are known smooth  while $\theta\in \R^l$ is an unknown  constant parameter vector. An adaptive Luenberger-like observer was proposed by  Huang et al in \cite{Huang}. 
{In our paper, the parameter $\theta$  may not always be constant, such as in the case of perturbations. Our sliding mode approach is efficient and straightforward, without the need to solve an additional ordinary differential equation, as seen in \cite{Huang}. Our results show that the convergence of the observer state to the actual state is exponential, while the rate of convergence in \cite{Huang} is unknown. Additionally, if the matrix function $f_2$ is bounded, the conditions for solvability of the associated LMI are significantly improved, leading to finite time convergence of the observation error under more favorable assumptions than in \cite{L3}. Our contributions also extend to the reduced-order case, where we have improved conditions. Lastly, we propose a new smooth approximation of the sliding mode technique that removes the chattering effect while still effectively managing uncertainty.}

{The paper is organized as follows. In Section 2, we revisit some key concepts. {In Section 3, we establish the well-posedness of the problem, addressing the existence and uniqueness of a solution. Additionally, we propose a sliding mode observer for (\ref{sysh}) and show its significant advantages over the Luenberger-like observer discussed in \cite{Huang}}. Section 4 focuses on the reduced-order observer. In Section 5, we present a new and refined version of the sliding mode technique. Numerical examples to support the results are provided in Section 6. Finally, Section 7 concludes the paper and highlights future possibilities.}
\section{{Notations} and mathematical background}
We denote the  scalar product and the corresponding norm of Euclidean spaces  by $\langle\cdot,\cdot\rangle$ and $\|\cdot\|$  respectively. 
A matrix $P\in \R^{n\times n}$  is called positive definite,  written  $P> 0$, if there exists $a>0$ such that 
$$
\langle Px,x \rangle \ge a\|x\|^2,\;\;\forall\;x\in {\R^n}.
$$
The sign  and Sign functions in $\R^n$ are defined by 
$$
{\rm sign}(x)=({\rm sign}(x_1), {\rm sign}(x_2), \ldots, {\rm sign}(x_n))
$$
and
$$
{\rm Sign}(x)= \left\{
\begin{array}{l}
\frac{x}{\Vert x\Vert} \;\;\;\;{\rm if}\;\; \;\;x\neq 0\\ \\
\ball \;\;\;\;\;\;\; {\rm if}\;\;\;\;\;x=0,
\end{array}\right.
$$
where $\ball$ denotes the unit ball in $\R^n$. Both functions coincides in $\R$, which are widely used in sliding mode technique (see, e.g., \cite{Orlov,Shtessel} and the references therein). Sign function is usually used in Walcott and Zak observer
 (see, e.g., \cite{Spurgeon,Wz,Xiang}). It seems to be inherited from the Filippov approach to the meaning of differential equations with discontinuous right-hand side.
 \begin{figure}[h!]
\begin{center}
\includegraphics[scale=0.7]{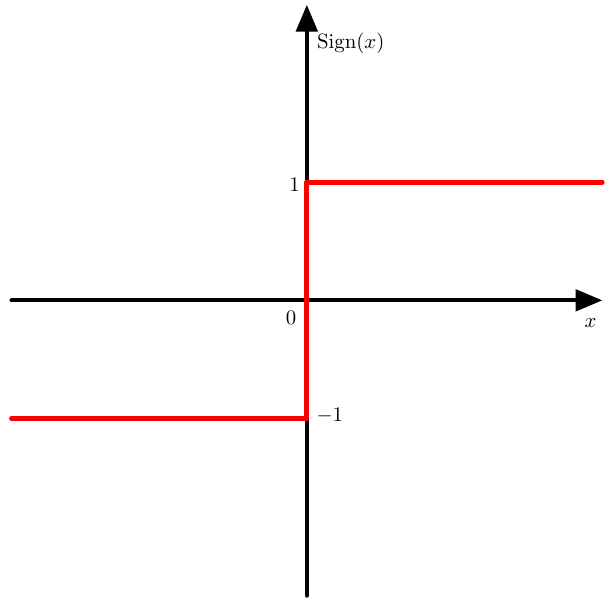}
\caption{Sign function in $\R$}
\end{center}
\label{luref}
\end{figure}

A set-valued mapping $\mathcal{F}: \R^m \rightrightarrows \R^m$ is called $\it{monotone}$ if for all $x,y\in \R^m$ and $x^*\in \mathcal{F}(x),y^*\in \mathcal{F}(y)$, one has 
$$\langle x^*-y^*,x-y\rangle \ge 0.$$ 
 Furthermore, $\mathcal{F}$  is called $\it{maximal \;monotone}$ if there is no monotone  operator $\mathcal{G}$ such that the graph of $\mathcal{F}$ is contained strictly in  the graph of $\mathcal{G}.$\\

\begin{lemma}\label{cct} \cite{L3}
Let $F\in \R^{p\times n}$ be a full row rank matrix ($p\le n$) and $P\in \R^{n\times n}$ be a symmetric positive definite matrix. If $x\in {\rm im}(P^{-1}F^T) $ then 
\beq
F^T(FP^{-1}F^T)^{-1}Fx=Px,
\eeq
where ${\rm im}(A) $ denotes the range of $A$.
\end{lemma}
Finally, let us recall the concept of orbital derivative. Let $x:[0,+\infty)\to \R^n$ be an absolutely continuous function, $V:  \R^n\to \R$ and $W(t):=V(x(t))$. Then the orbital derivative of $V$ along $x(\cdot)$ is defined by: $\dot{V}=\frac{dW}{dt}$.
\section{Well-posedness and the convergence analysis}
In this section, we will propose a sliding mode observer for the system (\ref{sysh}) under the following assumptions:\\

\noindent \textbf{Assumption 1: } The set-valued operator $\mathcal{F}: \R^n\rightrightarrows \R^n$ is a  monotone, upper semi-continuous with non-empty, closed convex and bounded values. \\

\noindent \textbf{Assumption 2: } The  continuous functions $f_1:\R^n\times \R^r\to \R^n$ and $f_2: \R^n\times \R^r\to \R^{n\times l}$ are Lipschitz continuous w.r.t $x$, i. e., there exist $L_1>0$ and $L_2>0$ such that for all $x_1,x_2\in \R^n, u\in \R^r$, we have 
$$
\Vert f_1(x_1,u)-f_1(x_2,u)\Vert \le L_1 \Vert x_1-x_2\Vert \;\;{\rm and}\;\;\Vert f_2(x_1,u)-f_2(x_2,u)\Vert \le L_2 \Vert x_1-x_2\Vert.
$$
\noindent \textbf{Assumption 3: } The unknown $\theta(t)$ is continuous and bounded by $L_3>0$.  \\

\noindent \textbf{Assumption 4: } Let $\gamma=L_1+L_2L_3$. There exist $\epsilon>0$, $P\in \R^{n\times n}>0$, $L\in \R^{n\times p}$, $K\in \R^{m\times p}$ and the matrix function $h: \R^n\times \R^r\to \R^{l\times p}$ such that 
\baq\label{ine}
&&P(A-LF)+(A-LF)^TP+\gamma P^2+\gamma I+\epsilon I \le 0,\\\label{pbc}
&& B^TP=C-KF,\\
&& f_2^T(x,u)P=h(x,u)F.\label{f2p}
\eaq
\noindent \textbf{Assumption 4': } the same to Assumption 4 but with $\gamma$ is replaced by $\gamma'=L_1$.
\begin{remark}\label{r1}\normalfont
i) From $(\ref{f2p})$, we have $f_2(x,u)=P^{-1}F^T h^T(x,u)$. It means that ${\rm im}(f_2(x,u))\in {\rm im}(P^{-1}F^T)$ for all $(x,u)$. \\
ii) Assumption 4 is widely used in the Control literature, which is based on the passivity and the matching property of disturbances (see, e.g., \cite{bh,Huang,L3}). 
\end{remark}
\begin{lemma}\label{lbh}
Let  $F\in \R^{p\times n}$ be a matrix with full row rank.  If $f_2$ is bounded, then $h$ is also bounded. Furthermore, if $f_2$ is Lipschitz continuous with respect to \;$x$ then $h$ is also Lipschitz continuous
with respect to $x$.
\end{lemma}
\begin{proof}
From Remark \ref{r1}, we imply that $h(x,u)=(FP^{-1}F^T)^{-1}Ff_2(x,u)$ and the condition follows. 
\end{proof}
\noindent The proposed sliding mode observer for (\ref{sysh})  is 
\beq \left\{
\begin{array}{lll}
\; \dot{\tilde{x}}=A\tilde{x} +B\tilde{\omega} -Le_y+f_1(\tilde{x},u)-\beta P^{-1}F^T\Vert h(\tilde{x},u)\Vert {\rm Sign}(e_y),\\
&&\\
\;\tilde{\omega} \in -\mathcal{F}(C\tilde{x}-Ke_y), \\
&&\\
\; \tilde{y}=F\tilde{x},
\end{array}\right.
\label{obs}
\eeq
where $\beta\ge L_3$. 
In the following theorem, we establish the well-posedness of the problem by proving the existence
and uniqueness of solutions for both the original system (\ref{sysh}) and the sliding mode observer (\ref{obs}),
under Assumptions 1--4.

\begin{theorem}\label{existence}
Suppose Assumptions   1--4 are satisfied. Then,
the following statements hold:\\
(i) The original system given by  (\ref{sysh}) and the sliding mode observer  given by (\ref{obs})
 both possess solutions.\\
(ii) If  $F$  is a matrix of full row rank, then the solution to the sliding mode observer (\ref{obs}) is unique.\\
(iii) If there exists a symmetric matrix $Q>0$ such that $B^TQ=C$, then the solutions of (\ref{sysh})   is unique.
\end{theorem}
\begin{proof}
(i)  The existence of solutions of both systems follows immediately since both systems can be reduced into first order differential inclusions where the right-hand side operators are upper semi-continuous and have non-empty,  convex and compact values (see for instance \cite{Ac,br,Huang}). 

(ii) Let $\tilde{x}_1$, $\tilde{x}_2$ be two solutions of $(\ref{obs})$ and $x$ is a solution of (\ref{sysh}). Let $\tilde{e}(t):=\tilde{x}_1(t)-\tilde{x}_2(t), \tilde{e}_{y1}:=F(\tilde{x}_1-x), \tilde{e}_{y2}:=F(\tilde{x}_2-x)$. Consider the Lyapunov function $V_1(\tilde{e})=\frac{1}{2}\langle P\tilde{e}, \tilde{e}\rangle$ and compute its orbital derivative
\baqn
\dot{V_1}&=&\langle P\tilde{e},\dot{\tilde{e}}\rangle\\
&=&\langle P\tilde{e},A\tilde{e}+B(\tilde{\omega}_1-\tilde{\omega}_2)+f_1(\tilde{x}_1,u)-f_1(\tilde{x}_2,u)\\
&-& \beta P^{-1}F^T(\Vert h(\tilde{x}_1,u)\Vert {\rm Sign}(\tilde{e}_{y1})-\Vert h(\tilde{x}_2,u)\Vert {\rm Sign}(\tilde{e}_{y2}))\rangle
\eaqn
The following inequalities hold true
$$
\langle P\tilde{e}, B(\tilde{\omega}_1-\tilde{\omega}_2)\rangle=\langle B^TP\tilde{e}, \tilde{\omega}_1-\tilde{\omega}_2\rangle=\langle (C-KF)\tilde{e}, \tilde{\omega}_1-\tilde{\omega}_2\rangle\le 0,
$$
$$
\langle P\tilde{e},A\tilde{e}+f_1(\tilde{x}_1,u)-f_1(\tilde{x}_2,u)\rangle \le (\Vert PA\Vert+L_1)\Vert \tilde{e} \Vert^2.
$$
Further calculations yield
\baqn
\Gamma:=&-&\langle P\tilde{e},  P^{-1}F^T(\Vert h(\tilde{x}_1,u)\Vert {\rm Sign}(\tilde{e}_{y1})-\Vert h(\tilde{x}_2,u)\Vert {\rm Sign}(\tilde{e}_{y2}))\rangle\\
&=& -\langle F\tilde{e},\Vert h(\tilde{x}_1,u)\Vert {\rm Sign}(\tilde{e}_{y1})-\Vert h(\tilde{x}_2,u)\Vert {\rm Sign}(\tilde{e}_{y2})\rangle\\
&=&- \Vert h(\tilde{x}_1,u)\Vert \langle F\tilde{e}, {\rm Sign}(\tilde{e}_{y1})-{\rm Sign}(\tilde{e}_{y2})\rangle-\langle F\tilde{e}, (\Vert h(\tilde{x}_1,u)\Vert -\Vert h(\tilde{x}_2,u)\Vert){\rm Sign}(\tilde{e}_{y2})\rangle\\
&\le & \Vert F \Vert L_h \Vert \tilde{e} \Vert^2,
\eaqn
where $L_h$ is the Lipschitz constant of $h$ (Lemma \ref{lbh}) and we have used the monotonicity of the ${\rm Sign}$ function. Therefore $\dot{V_1}\le \rho_1 \Vert \tilde{e} \Vert^2$ where $\rho_1:=\Vert PA\Vert+L_1+L_h \Vert F \Vert$, allowing us to draw the
conclusion by using the Gronwall's inequality.

(iii)  Let $x_1$ and $x_2$ be two solutions of $(\ref{sysh})$. Define
the error $e(t)$ as $e(t):=x_1(t)-x_2(t)$. We choose  the Lyapunov function $V_2(e)=\frac{1}{2}\langle Qe, e\rangle$ and
compute its orbital derivative
\baqn
\dot{V_2}&=&\langle Qe,\dot{e}\rangle\\
&=&\langle Qe,Ae+B(\omega_1-\omega_2)+f_1(x_1,u)-f_1(x_2,u)\\
&+& f_2(x_1,u)\theta(t)-f_1(x_2,u)\theta(t)\rangle.
\eaqn
By using the monotonicity of  $\mathcal{F}$, we obtain 
$$
\langle Qe,B(\omega_1-\omega_2) \rangle=\langle B^TQe,\omega_1-\omega_2\rangle=\langle Ce, \omega_1-\omega_2\rangle \le 0.
$$
Additionally, we have the following inequalities
$$
\Vert f_1(x_1,u)-f_1(x_2,u) \Vert \le L_1 \Vert e \Vert.
$$
\baqn
\Vert f_2(x_1,u)\theta(t)-f_2(x_2,u)\theta(t)\Vert\le L_2L_3 \Vert e \Vert.
\eaqn
Therefore, $\dot{V_2}\le \rho_2 \Vert e \Vert^2$ where $\rho_2:=\Vert QA\Vert+L_1+L_2L_3$ and  the conclusion follows.
\end{proof}
The following result confirms the exponential convergence of the observer state to the original state.
\begin{theorem}\label{th1}
Let Assumptions 1--4 hold. Then the observer state $\tilde{x}$ of (\ref{obs}) exponentially converges to the original state $x$ of (\ref{sysh}). 
\end{theorem}
\begin{proof}
Let $e=\tilde{x}-x, e_y=\tilde{y}-y=Fe$. From (\ref{sysh}) and (\ref{obs}), we have 
\baq\nonumber
\dot{e}&\in&(A-LF)e-B(\tilde{\omega}-{\omega})+f_1(\tilde{x},u)-f_1({x},u)\\
&-&\beta P^{-1}F^T\Vert h(\tilde{x},u)\Vert {\rm Sign}(e_y)-f_2(x,u)\theta. \label{der}
\eaq
Consider the Lyapunov function $V(e)=\langle Pe, e\rangle $. Then the orbital derivative of $V$ is 
\baq\nonumber
\dot{V}(e)&=&2\langle P\dot{e}, e\rangle=2\langle P(A-LF)e-PB(\tilde{\omega}-{\omega}), e\rangle\\
&+&2\langle P(f_1(\tilde{x},u)-f_1({x},u)), e \rangle -2\beta\Vert h(\tilde{x},u)\Vert \Vert e_y \Vert - 2\langle Pf_2(x,u)\theta,e\rangle.\label{dv}
\eaq
From (\ref{pbc}) and the monotonicity of $ \mathcal{F}$, we have 
\baq
\langle PB(\tilde{\omega}-{\omega}), e\rangle=\langle \tilde{\omega}-{\omega}), B^TP e\rangle
=\langle \tilde{\omega}-{\omega}, (C-KF) e\rangle\le 0. 
\label{mn}
\eaq
On the other hand
\baq
2\langle P(f_1(\tilde{x},u)-f_1({x},u)), e \rangle\le 2L_1 \Vert e\Vert \Vert Pe\Vert \le {L_1}(\Vert e\Vert ^2+\Vert Pe \Vert^2)\label{et1}
\eaq
and
\baq \nonumber
Y&:=&-2\beta\Vert h(\tilde{x},u)\Vert \Vert e_y \Vert - 2\langle Pf_2(x,u)\theta,e\rangle\\\nonumber
&=&-2\beta\Vert h(\tilde{x},u)\Vert \Vert e_y \Vert -  2\langle P(f_2(x,u)-f_2(\tilde{x},u))\theta,e\rangle-2\langle Pf_2(\tilde{x},u)\theta,e\rangle\\\nonumber
&\le&-2\beta\Vert h(\tilde{x},u)\Vert \Vert e_y \Vert +2L_2L_3 \Vert e\Vert \Vert Pe \Vert-2\langle \theta, h(\tilde{x},u)e_y\rangle \\
&\le&{L_2L_3}(\Vert e\Vert ^2+\Vert Pe \Vert^2).\label{et2m}
\eaq
Let $\alpha_{max}$ and $\alpha_{min}$ be the largest eigenvalue and the smallest  eigenvalues of $P$, respectively. 
From (\ref{ine}),  (\ref{dv}), (\ref{mn}), (\ref{et1}) and (\ref{et2m}), we have $$\frac{dV}{dt}\le-\epsilon \Vert e\Vert ^2 \le -\frac{\epsilon}{\alpha_{max}}V(t).$$
 Using Gronwall's inequality, we obtain 
$$
{\alpha_{min}\Vert e\Vert ^2}\le V(t)\le {\rm exp}({\frac{-\epsilon t}{\alpha_{max}}})V(t_0).
$$
Hence, $\Vert e \Vert \le {\rm exp}({\frac{-\epsilon t}{2\alpha_{max}}})\sqrt{\frac{V(t_0)}{{\alpha_{min}}}}$ and the conclusion follows. 
\end{proof}
\begin{remark}\normalfont
{(i) According to \cite{Huang}, under similar conditions, our implementation of a sliding mode observer leads to an exponential convergence of the observer state, whereas the adaptive observer in the same study only achieves convergence without a defined rate. Furthermore, our method avoids the need to solve a high-dimensional additional ODE, which can be computationally expensive, as seen in \cite{Huang}.}\\
{(ii) In Assumption 4, if $h$ is bounded, then $\gamma=L_1+L_2L_3$ can be replaced by just $L_1$, which is usually much smaller than $\gamma$. This makes it easier to satisfy (\ref{ine})-(\ref{f2p}) and enables us to solve a wider class of set-valued Lur'e systems that cannot be tackled using the method in \cite{Huang}. Furthermore, our technique also ensures finite time convergence of the observation error $e_y$ to zero. A bounded matrix function $h$ can be guaranteed if $f_2$ is bounded and $F$ is full row rank (Lemma \ref{lbh}), as one can always choose a full row rank matrix $F$ for the output $y=Fx$ without losing any information.}
\end{remark}
\begin{theorem}\label{tm2}
 Let Assumptions 1, 2, 3, 4' hold and suppose that $h(x,u)$ is upper-bounded by some constant $L_4>0$. Then we  can still ensure the exponential convergence to the original state by using the  observer 
\beq \left\{
\begin{array}{lll}
\; \tilde{x}=A\tilde{x} +B\tilde{\omega} -Le_y+f_1(\tilde{x},u)+\beta P^{-1} F^T{\rm Sign}(e_y),\\
&&\\
\;\tilde{\omega} \in -\mathcal{F}(C\tilde{x}-Ke_y), \\
&&\\
\; \tilde{y}=F\tilde{x} 
\end{array}\right.
\label{obsn}
\eeq
where $\beta > L_3L_4 $. Furthermore,  if  the matrix $F\in \R^{p\times n}$ has full row rank  and ${\rm im}(B)\subset {\rm im}(P^{-1}F)$ then  the observation error $e_y$ converges in explicitly finite time. 
\end{theorem}
\begin{proof}
Under the new assumption, in (\ref{et2m}) we have 
\baq \nonumber
-\beta \Vert e_y \Vert - \langle Pf_2(x,u)\theta,e\rangle
=-\beta\Vert e_y \Vert -\langle \theta, h({x},u)e_y\rangle 
\le -(\beta-L_3L_4)\Vert e_y\Vert \le 0 \label{et2}
\eaq
and similarly one obtains the exponential convergence of the observer state. 

To attain the  finite time convergence of the observation error, we consider the new Lyapunov function $W=\frac{1}{2}\langle e_y,(FP^{-1}F^T)^{-1} e_y \rangle$. Then 
$$
\frac{dW}{dt}=\langle (FP^{-1}F^T)^{-1} \dot{e}_y, e_y\rangle.
$$
 From (\ref{der}), we have 
\baq\nonumber
\dot{e}_y&=&F(A-LF)e-FB(\tilde{\omega}-{\omega})+F\Big(f_1(\tilde{x},u)-f_1({x},u)\Big)\\
&-&\beta FP^{-1}F^T {\rm Sign}(e_y)-Ff_2(x,u)\theta. \label{dery}
\eaq
Since ${\rm im}(B)\subset {\rm im}(P^{-1}F)$, using Lemma 1, we have 
$$
F^T(FP^{-1}F^T)^{-1}FB=PB
$$
and hence
\baqn
&&\langle(FP^{-1}F^T)^{-1}FB(\tilde{\omega}-{\omega}), e_y\rangle 
=\langle F^T(FP^{-1}F^T)^{-1}FB(\tilde{\omega}-{\omega}), e \rangle \\
&=&\langle PB(\tilde{\omega}-{\omega}), e \rangle =\langle \tilde{\omega}-{\omega}, B^TP e\rangle
=\langle \tilde{\omega}-{\omega}, (C-KF) e\rangle\le 0.
\eaqn
On the other hand, from (\ref{f2p}) we deduce that $f_2=P^{-1}F^T h^T$ and thus 
\beq\label{newt}
(FP^{-1}F^T)^{-1}Ff_2(x,u)\theta=(FP^{-1}F^T)^{-1}(FP^{-1}F^T)h^T(x,u)\theta=h^T(x,u)\theta.
\eeq
Since $\Vert e \Vert \le {\rm exp}({\frac{-\epsilon t}{2\alpha_{max}}})\sqrt{\frac{V(t_0)}{{\alpha_{min}}}}$ as in the proof of Theorem \ref{th1}, we can find some $t_1>0$ such that for all $t\ge t_1$, one has 
$$
\Vert (FP^{-1}F^T)^{-1}\Big (F(A-LF)e+F\big(f_1(\tilde{x},u)-f_1({x},u)\big)\Big)\Vert\le \frac{\sigma}{2}
$$
where $\sigma=\beta-L_3L_4>0$. Indeed, we can choose 
$$
t_1:=\frac{2\alpha_{max}}{\epsilon}\ln\Big(\frac{2\Vert (FP^{-1}F^T)^{-1}F\Vert(\Vert  A-LF\Vert+L_1)}{\sigma}\sqrt{\frac{V(t_0)}{{\alpha_{min}}}}\Big).
$$
Then for all $t\ge t_1$, we have 
\baqn
\frac{dW}{dt}&\le& \frac{\sigma}{2} \Vert e_y \Vert - (\beta -\Vert h^T(x,u)\theta \Vert) \Vert e_y \Vert\le \frac{\sigma}{2} \Vert e_y \Vert - (\beta -L_3L_4) \Vert e_y \Vert\\
&=&-\frac{\sigma}{2} \Vert e_y \Vert\le -\kappa  \sqrt{W(t)},
\eaqn
where $\kappa:=\sigma/\sqrt{2\gamma_{max}}$ and $\gamma_{max}$ is the largest eigenvalue of $(FP^{-1}F^T)^{-1}$ since $\Vert e_y \Vert\ge \sqrt{\frac{2W(t)}{{\gamma_{max}}}}.$ Suppose that $W(t)>0$ for all $t\ge t_1$ then we have
$$
\frac{W'}{2\sqrt{W}}\le -\frac{\kappa}{2}
$$
and thus
$$
\sqrt{W(t)}-\sqrt{W(t_1)}\le -\frac{\kappa}{2}(t-t_1)\to-\infty \;\;{\rm as}\;t\to \infty,
$$
a contradiction. Let $t_f\ge t_1$ be the first time such that $W(t_f)=0$, then we deduce that $W(t)=0$ for all $t\ge t_f$ since $W$ is non-negative and decreasing. It means that  $e_y$ converges to $0$ in finite time.  Similarly as above we have 
$$
-\sqrt{W(t_1)}=\sqrt{W(t_f)}-\sqrt{W(t_1)}\le -\frac{\kappa}{2}(t_f-t_1)
$$
and hence
$$
t_f\le t_1+2\sqrt{W(t_1)}/\kappa,
$$
which completes the proof of the theorem.
\end{proof}
\begin{remark}\normalfont
{(i)  Based on  the observation (\ref{newt}), we can have a significantly better estimation for the gain $\beta$ than in \cite{L3}. Indeed, in the current paper the gain $\beta$ does not change to obtain the  finite time convergence of the observation error from the  exponential convergence of the observer state. In addition, we can also provide an explicit estimation for $t_f$.}\\
{(ii) If Assumption 1 is substituted by the maximal monotone property of the set-valued $\mathcal{F}$, similar results can still be obtained. The existence and uniqueness of solutions for the original  system (\ref{sysh}) remains guaranteed, as shown in references such as \cite{Ac,br}. Existence  of solutions to the observer systems (\ref{obs}),  (\ref{obsn}) can  be also obtained by rewriting these systems into the form $\dot{x}\in -\mathcal{A}x+G(t,x)$ where $\mathcal{A}$ is a maximal monotone operator and $G$ is a upper semi-continuous set-valued function w.r.t $x$ with non-empty convex compact values (see, e.g., \cite{AB,bh}). The uniqueness can be proved similarly as the proof of Theorem \ref{existence}. It's worth mentioning that the normal cone operator $N_C$, which is associated to a nonempty closed convex set $C$, is an important maximal monotone set-valued mapping in mechanical and electrical engineering that does not satisfy  the boundedness requirement in Assumption 1.}
\end{remark}
\section{Reduced-order  observer}
Suppose that the given matrices and matrix-functions can be decomposed as follows
$$x=\left( \begin{array}{ccc}
x_1 \\ \\
x_2 
\end{array} \right), \;\;A=\left( \begin{array}{cc}
A_{11} &\;\; A_{12} \\ \\
A_{21} & \;\; A_{22}
\end{array} \right), B=\left( \begin{array}{ccc}
B_1 \\ \\
B_2 
\end{array} \right), C= (C_1\;\;C_2), F= (F_q \;\;0)
$$
$$
P=\left( \begin{array}{cc}
P_{11} &\;\; P_{12} \\ \\
P_{21} & \;\; P_{22}
\end{array} \right), \;\;f_1(x,u)=\left( \begin{array}{ccc}
f_{11}(x,u) \\ \\
f_{12}(x,u) 
\end{array} \right), \;\;f_2(x,u)=\left( \begin{array}{ccc}
f_{21}(x,u) \\ \\
f_{22}(x,u) 
\end{array} \right),
$$
where $F_q\in R^{q\times q}$ is an invertible matrix and the following is satisfied:\\

\noindent \textbf{Assumption 4'': }  There exist $\epsilon>0$, $Q\in \R^{(n-q)\times (n-q)}>0$,   $P_{22}\in \R^{(n-q)\times (n-q)}$ invertible, $P_{21}\in \R^{(n-q)\times q}$ such that 
\baq\label{inen}\nonumber
&&Q(A_{22}+KA_{12})+(A_{22}+KA_{12})^TQ+L_1 Q(KK^T+I_{n-q})Q\\
&&\;\;\;\;\;\;\;\;\;\;\;\;\;\;\;\;\;\;\;\;\;\;\;\;\;\;\;\;\;\;\;\;\;\;\;\;\;\;\;\;\;\;\;\;\;\;\;\;\;\;\;\;\;\;\;\;\;\;+(L_1+\epsilon) I_{n-q} \le 0,  \label{diss}\\
&& (B_2+KB_1)^TQ=C_2, \label{trans}\\
&&\Big( P_{21}\;\;P_{22}\Big)f_2(x,u)=0, \label{f2pn}
\eaq
where $K=P_{22}^{-1}P_{21}$.\\

Note that (\ref{sysh}) can be rewritten as follows
\beq \left\{
\begin{array}{lll}
\; \dot{x}_1=A_{11}x_1 +A_{12}x_2 +B_1\omega +f_{11}\Big(\left( \begin{array}{ccc}
x_1 \\ \\
x_2 
\end{array} \right),u\Big)+f_{21}\Big(\left( \begin{array}{ccc}
x_1 \\ \\
x_2 
\end{array} \right),u\Big)\theta\\
\; \dot{x}_2=A_{21}x_1 +A_{22}x_2 +B_2\omega +f_{12}\Big(\left( \begin{array}{ccc}
x_1 \\ \\
x_2 
\end{array} \right),u\Big)+f_{22}\Big(\left( \begin{array}{ccc}
x_1 \\ \\
x_2 
\end{array} \right),u\Big)\theta\\
&&\\
\;\omega \in -\mathcal{F}(C_1x_1+C_2x_2) \\
&&\\
\; y=F_qx_1 .
\end{array}\right. \label{syshdis}
\eeq
 Using (\ref{f2pn}), we have 
\beq \left\{
\begin{array}{lll}
\; \dot{z}&=&(A_{22}+KA_{12})z +(B_2+KB_1)\omega  +[(A_{21}+KA_{11})-(A_{22}+KA_{12})K]x_1\\
 &+&(K \;\;I_{n-q})f_1\Big(\left( \begin{array}{ccc}
x_1 \\ \\
z-Kx_1
\end{array} \right),u\Big)\\
&&\\
\;\omega &\in& -\mathcal{F}(C_2z+(C_1-C_2K)x_1), \\
&&\\
\; x_2&=&z-Kx_1.
\end{array}\right. 
\eeq
  The adaptive observer is 
\beq \left\{
\begin{array}{lll}
\; \dot{\tilde{z}}&=&(A_{22}+KA_{12})\tilde{z} +(B_2+KB_1)\tilde{\omega}  +[(A_{21}+KA_{11})-(A_{22}+KA_{12})K]x_1\\
 &+&(K \;\;I_{n-q})f_1\Big(\left( \begin{array}{ccc}
x_1 \\ \\
\tilde{z}-Kx_1 
\end{array} \right),u\Big)\\
&&\\
\;\tilde{\omega} &\in& -\mathcal{F}(C_2\tilde{z}+(C_1-C_2K)x_1), \\
&&\\
\; \tilde{x}_2 &=&\tilde{z}-Kx_1.
\end{array}\right. \label{obseu}
\eeq

\begin{theorem}\label{reduce}
Let Assumptions 1, 2, 3, 4'' hold. Then (\ref{obseu}) is a reduced-order observer of $(\ref{sysh})$, i.e., $\displaystyle{\lim_{t\to \infty}(x_2(t)-\tilde{x}_2(t))}=0$.
\end{theorem}
\begin{proof}
Let $e_z=\tilde{z}-{z}$. Then we have
\beq \left\{
\begin{array}{lll}
\; \dot{e}_{z}&=&(A_{22}+KA_{12})e_{z} +(B_2+KB_1)(\tilde{\omega}-\omega) \\
 &+&(K \;\;I_{n-q})\Big[f_1\Big(\left( \begin{array}{ccc}
x_1 \\ \\
\tilde{z}-Kx_1 
\end{array} \right),u\Big)-f_1\Big(\left( \begin{array}{ccc}
x_1 \\ \\
{z}-Kx_1
\end{array} \right),u\Big)\Big]\\
&&\\
\;{\omega} &\in& -\mathcal{F}(C_2{z}+(C_1-C_2K)x_1).\\
&&\\
\;\tilde{\omega} &\in& -\mathcal{F}(C_2\tilde{z}+(C_1-C_2K)x_1).
\end{array}\right. 
\eeq
Let us consider the Lyapunov function $W(e_z)=\langle Q e_z,e_z \rangle$, then $\dot{W}(e_z)=2\langle Q \dot{e}_z,e_z \rangle$. 
From (\ref{trans}) and the monotonicity of $\mathcal{F}$, we have
$$
\langle Q(B_2+KB_1)(\tilde{\omega}-\omega), e_z\rangle=\langle\tilde{\omega}-\omega, (B_2+KB_1)^TQe_z\rangle=\langle\tilde{\omega}-\omega, C_2e_z\rangle\le 0.
$$
On the other hand
\baqn
&&\langle 2Q(K \;\;I_{n-q})\Big[f_1\Big(\left( \begin{array}{ccc}
x_1 \\ \\
\tilde{z}-Kx_1 
\end{array} \right),u\Big)-f_1\Big(\left( \begin{array}{ccc}
x_1 \\ \\
{z}-Kx_1
\end{array} \right),u\Big)\Big], e_z \rangle\\
&\le& 2L_1\Vert (K \;\;I_{n-q})^TQ e_z\Vert\Vert e_z\Vert\le L_1e_z^TQ(K \;\;I_{n-q}) (K \;\;I_{n-q})^TQe_z+L_1e_z^Te_z\\
&\le&L_1e_z^TQ(KK^T+I_{n-q})Qe_z+L_1e_z^Te_z.
\eaqn
Combining with (\ref{diss}), similarly as in the proof of Theorem \ref{th1}, we have $\dot{W}(e_z)\le -\epsilon \Vert e_z\Vert^2$. It deduces that $e_z$ converges to zero exponentially and the conclusion follows. 
\end{proof}
\begin{remark}\normalfont
(i)  Note that if we have  (\ref{f2p}), then $Pf_2(x,u)=F^Th(x,u)$ and thus \\$\Big( P_{21}\;\;P_{22}\Big)f_2(x,u)=0$. It would be interesting to improve or remove the condition (\ref{f2pn}). \\
(ii) When the matrices are decomposable, it is more effective to provide assumptions directly with the new lower-dimension matrices. Note that Assumption 4'' is strictly weaker than Assumption 4  even when $Q=P_{22}$ \cite{Huang}. It is remarkable that $Q>0$ can be different from $P_{22}$ and it is unnecessary to require that $P>0$ but only invertible $P_{22}$. {This enhancement significantly broadens the potential
applications. For simplicity one can choose  $P_{21}=0$, which would consequently result in $K=0$. It's noteworthy that  in Assumption 4'', only $L_1$  is involved, just as in Assumption 4' used for sliding mode observer (\ref{obsn}). However, one of the drawbacks of the reduced-order observer is the necessity to perform some
linear transformations  if $F$ is given in a general form.} {Additionally, despite satisfying Assumptions 1--3 and 4'', the numerical convergence of the reduced-order observer (\ref{obseu}) may fail in certain sensitive cases, as illustrated in Example 3, due to its reliance on the approximate output $x_1$ from the original system. In contrast, the sliding mode observer is known for its enhanced robustness.
}
\end{remark}
\section{A new continuous approximate of the sliding mode technique}
{Although the sliding mode method is effective, it has a persistent issue: the chattering effect caused by the discontinuity of the Sign function. To eliminate this issue, the Sign function in $\R$ is typically approximated by the ``sigmoid function"  (as noted in \cite{Shtessel})}
$$
 {\rm Sign(x)}\approx \frac{x}{\vert x \vert+\epsilon }
$$
 or by  \cite{Shokouhi}
$$
 {\rm Sign(x)}\approx \frac{x}{\sqrt{ x^2+\epsilon }}
$$
for some small fixed $\delta>0$. Other continuous approximates can be also found in \cite{Shokouhi}. 
 \begin{figure}[h!]
\begin{center}
\includegraphics[scale=0.2]{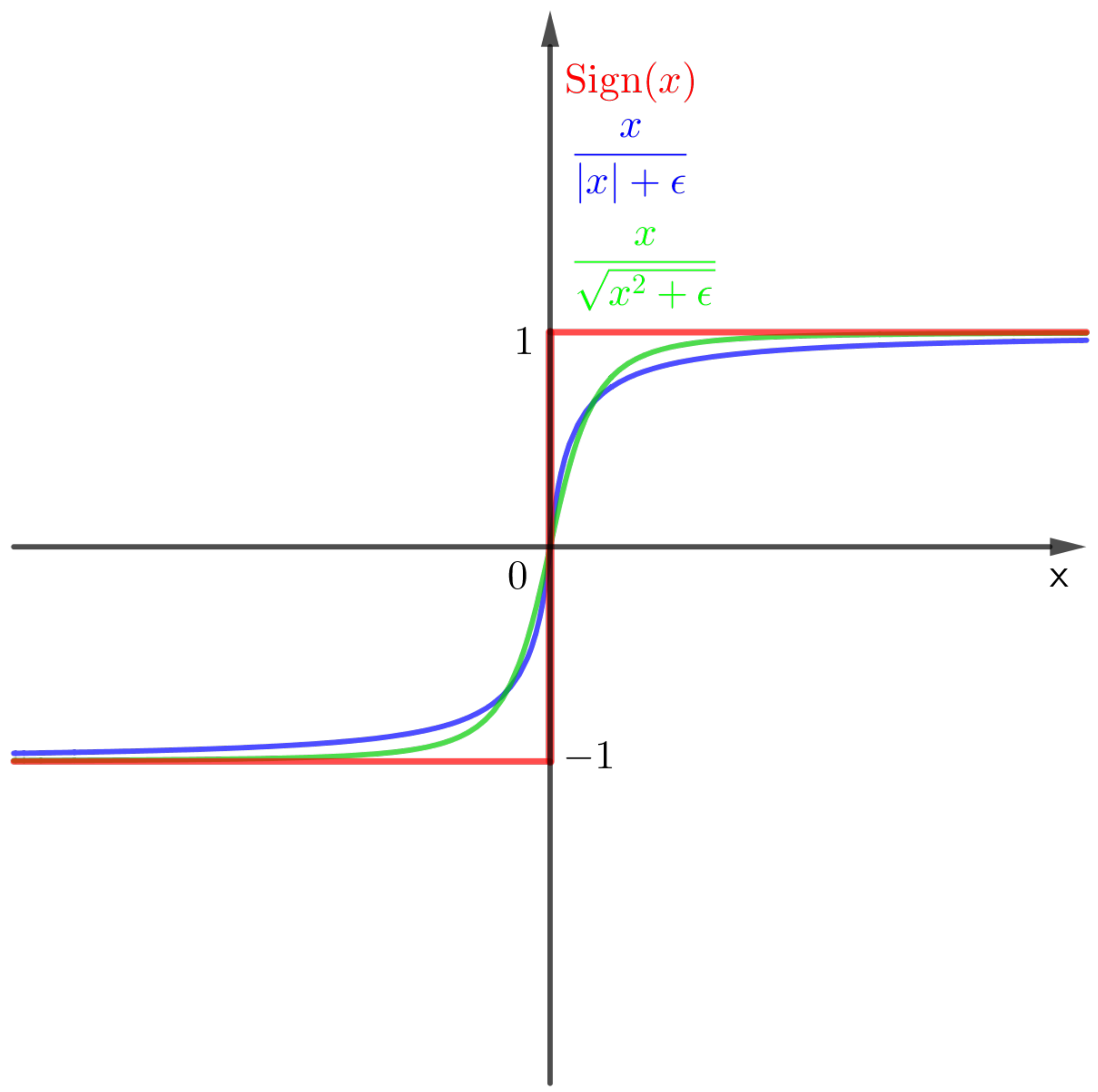}
\caption{Sign function in $\R$ and its approximations}
\end{center}
\label{luref}
\end{figure}

In this paper, we provide a new  smooth approximate of the set-valued function {\rm Sign} by using  a time-dependent {guiding} function. Given a continuous  {guiding}  function $\delta: \R^+\to\R^+$, we define the function ${\rm Sign}_\delta: \R^+ \times \R^n \to \R^n$ as follows 
  \begin{equation}
{\rm Sign}_\delta(t,x)=\left\{
\begin{array}{l}
\frac{x}{\Vert x\Vert}-\frac{1- {\Vert x\Vert}/{\delta(t)}}{(1+M\Vert x\Vert)^N}\frac{x}{\Vert x\Vert} \;\;{\rm if}\;\; x\neq 0\;{\rm and }\;\Vert x\Vert \le \delta(t), \\ \\
\frac{x}{\Vert x\Vert} \;\; \;\;{\rm if}\;\;\;\;\Vert x\Vert > \delta(t),\\ \\
0 \;\;\;\;\; \;\;\;{\rm if}\;\;\;\;\;x=0,
\end{array}\right.
\label{appr}
\end{equation}
for some $M, N>0$  and the  function $\delta(\cdot)$ is non-increasing. For example, we can choose $\delta(t)={\rm exp}(-k_1t-k_2)$ for some $k_1>0, k_2>0$.   
{It can be seen that the ${\rm Sign}_\delta(t,x)$ function is a continuous  function with respect to both time $t$ and state $x$. When the magnitude of $x$ is greater than $\delta(t)$, the function ${\rm Sign}_\delta(t,x)$ becomes equal to ${\rm Sign}(x)$. These properties make it effective in reducing the chattering effect while still handling uncertainty. In comparison, the norm of the sigmoid function $\frac{x}{\vert x \vert+\epsilon }$ is always less than $1$ which does not deal with the uncertainty entirely.  The sigmoid function only leads to convergence of the state to an approximate region of the sliding surface, as reported, e.g., in  \cite{Shtessel}. The suitable choice of the guiding function $\delta$ can remarkably reduce the chattering effect. In practice, one  can choose $\delta(t)=\max\{10^{-3},{\rm exp}(-k_1t-k_2)\}$ to avoid  very small values when $t$ is large. This can be seen in Example 1, Figures 3--7.}

 \section{Numerical Examples}
 
  In this section we provide some numerical examples to show the effectiveness of our approach. 
  \begin{example}\normalfont

{To test the effectiveness of our new approximation method, we will analyze a simple one-dimensional system: $\dot{x}\in \mu -L {\rm Sign}(x)$, where $\mu$ is an uncertainty with a constraint $\vert \mu\vert < L$. As an example, let us set $\mu=3\sin x$ and $L=4$. It is known that the state variable $x$ will converge to zero in a finite time. To evaluate the numerical simulations, we use the explicit scheme with an initial value of $x_0=0.1$ (see Figure 3).}
{The {\rm Sign} function  creates a chattering effect (as seen in the state $x_1$). However, by replacing the ${\rm Sign}$ function with ${\rm Sign}_\delta(t,x)$ defined in equation (\ref{appr}),  using a guiding function $\delta(t)=e^{-0.5t}$ and  $M=1, N=3$, the chattering effect can be eliminated and the convergence to zero is achieved (as seen in the state $x_3$). The performance of this approximation is better than using the sigmoid function with $\epsilon=10^{-3}$ (as seen in the state $x_2$).}
\begin{figure}[h!]
\begin{center}
\includegraphics[scale=0.49]{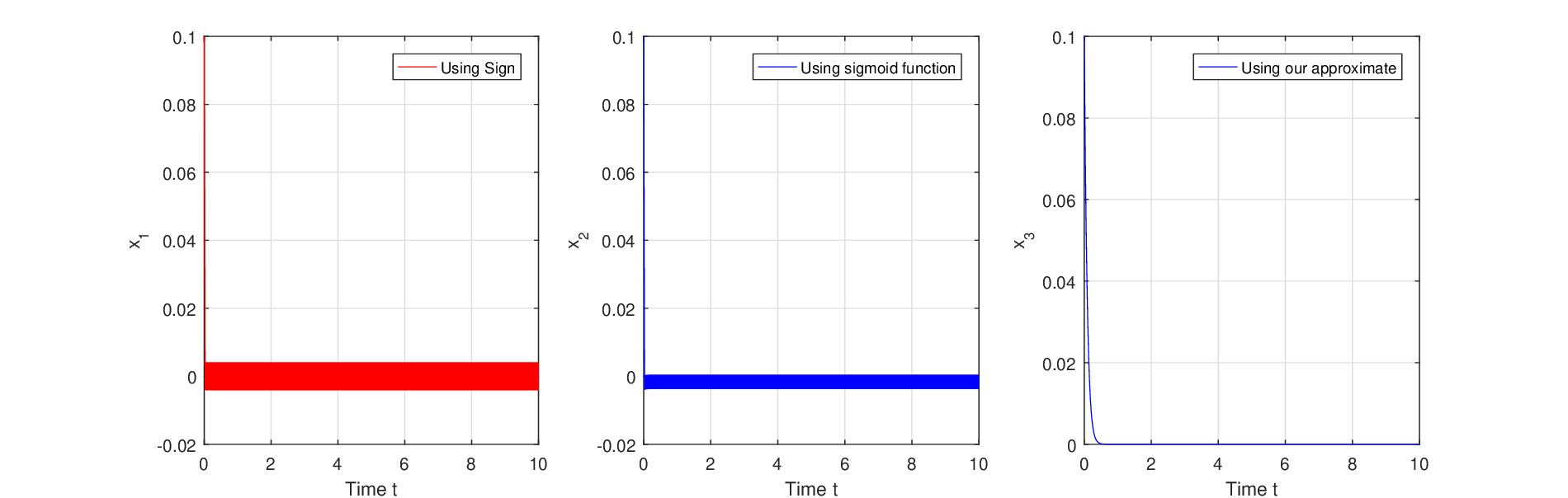}
\caption{An effective continuous approximation of Sign function in $\R$}
\end{center}
\label{luref}
\end{figure}

 \begin{figure}[h!]
\begin{center}
\includegraphics[scale=0.42]{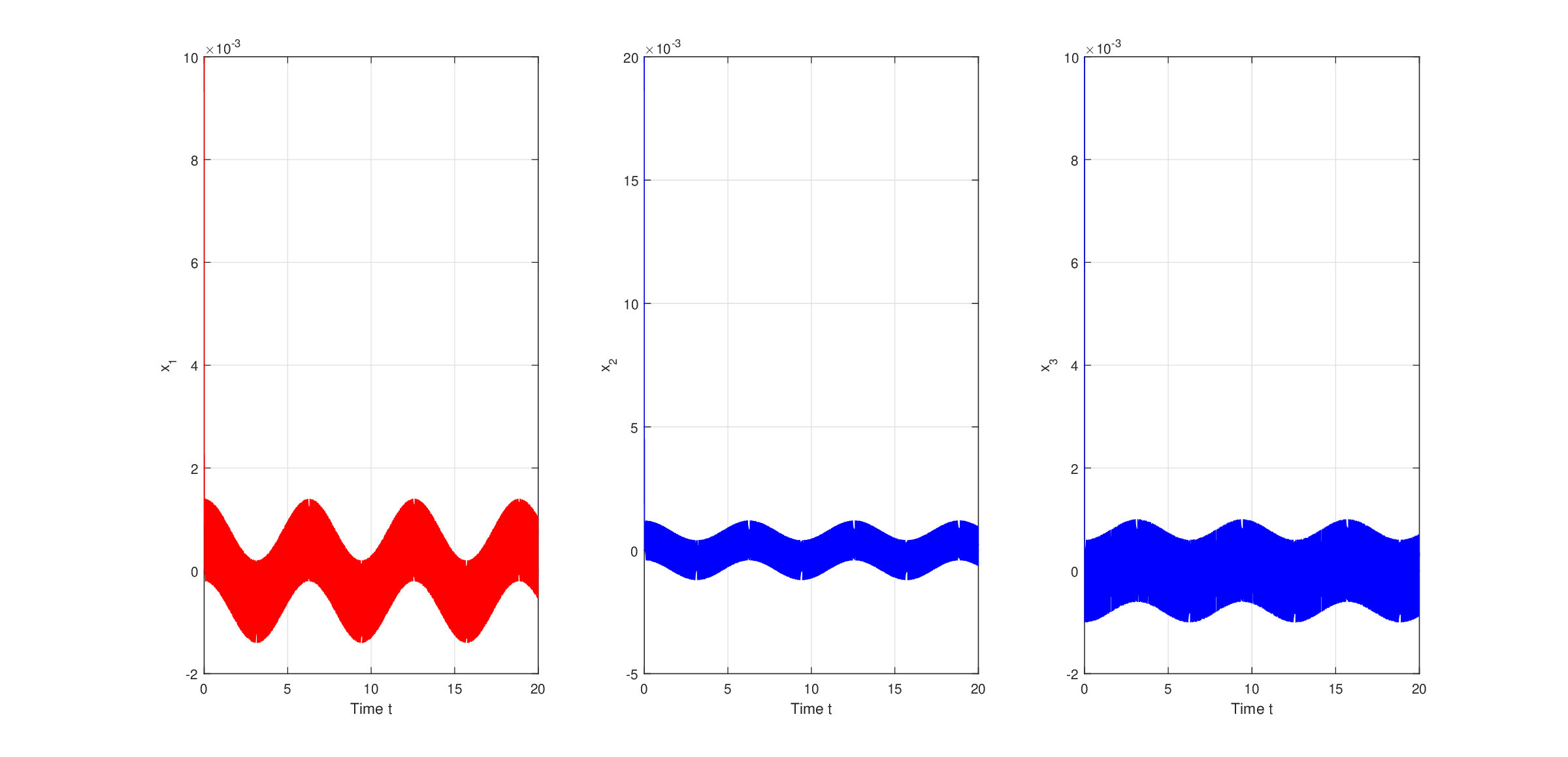}
\caption{The sytem state in $\R^3$ using sign function} 
\end{center}
\end{figure}

Similarly, we can consider the same system $\dot{x}\in \mu -Lu$ in $\R^3$ where the control $u$ can be ${\rm sign(x)}, \;{\rm Sign(x)}$ or certain continuous approximation. Let us consider  $\mu=(3 \;\;2\;\;-1)^T\cos t, L=40$ with the initial condition $x_0= (0.01\;\;0.02\;\;0.01)^T$. One can see in Figures 4 and 5 that the Sign function reduces chattering better than the sign function since  the Sign function is only discontinuous at zero. On the other hand,  ${\rm Sign}_\delta(t,x)$   has the best performance in reducing chattering (Figure 7). It confirms  that our time-dependent approximation  worths considering as a good alternative besides  existing fixed approximations and merits further investigations.  
 \begin{figure}[h!]
\begin{center}
\includegraphics[scale=0.42]{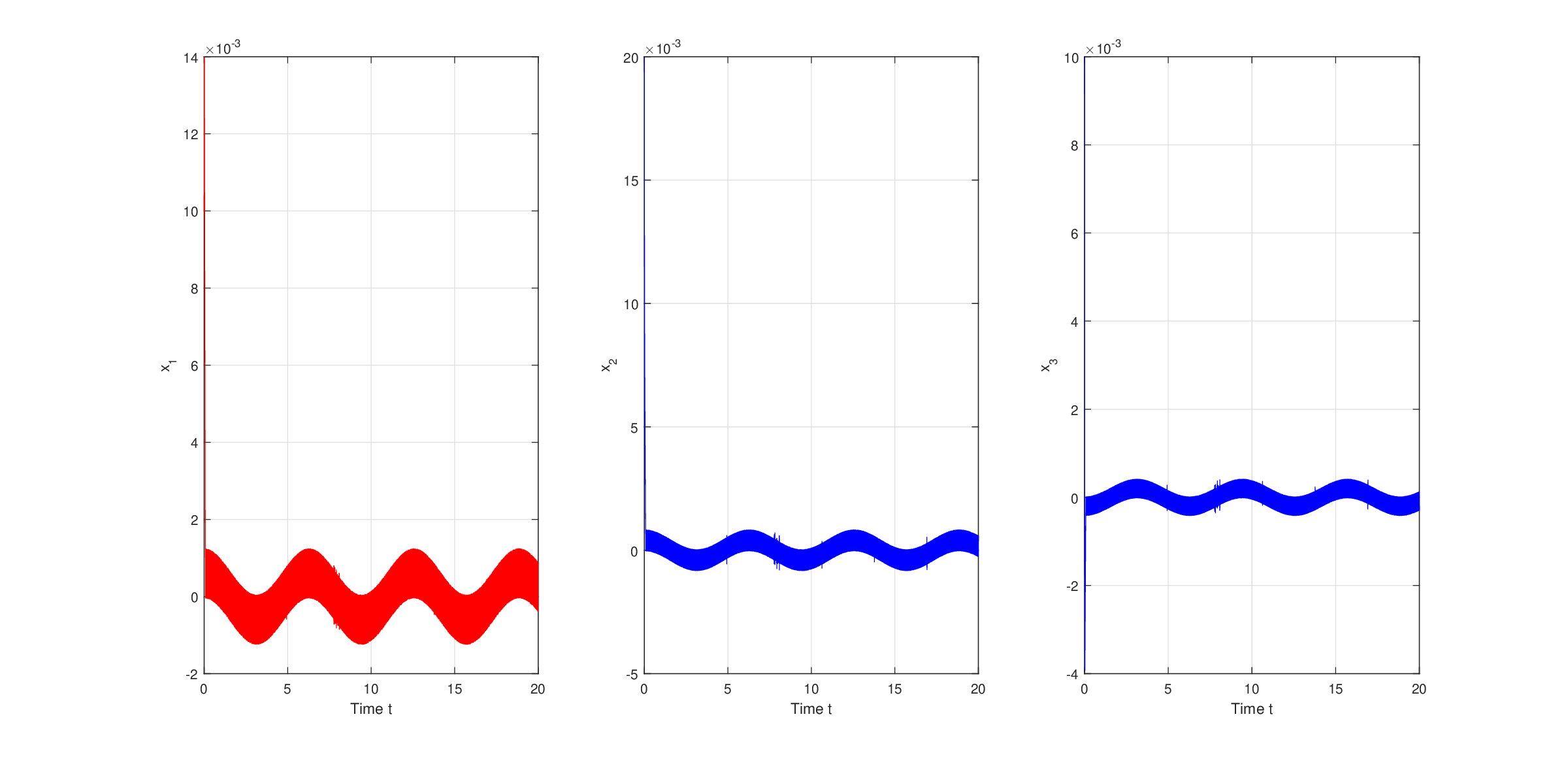}
\caption{The sytem state in $\R^3$ using Sign function} 
\end{center}
\end{figure}
 \begin{figure}[h!]
\begin{center}
\includegraphics[scale=0.42]{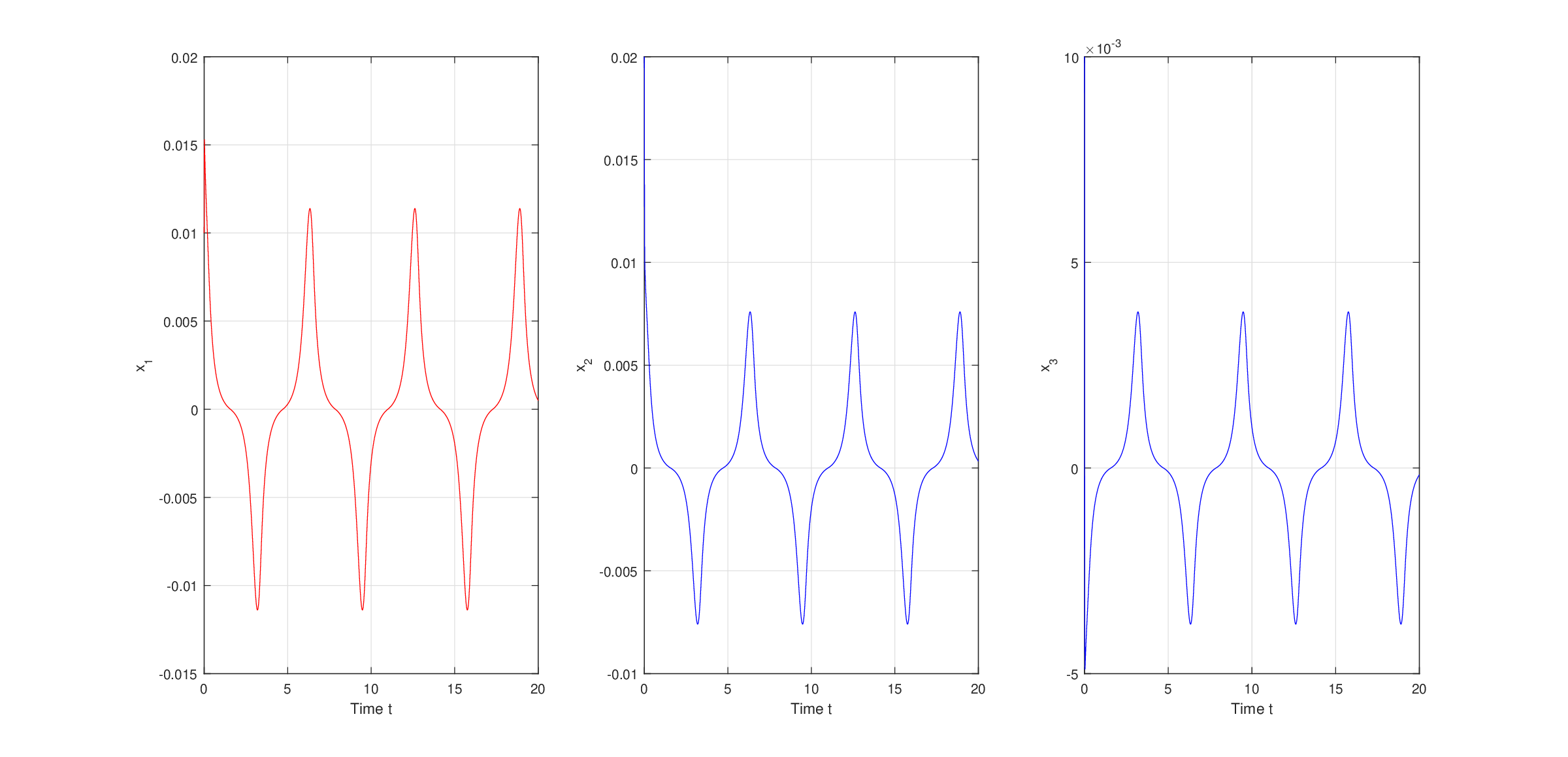}
\caption{The sytem state in $\R^3$ using sigmoid function with $\epsilon=10^{-3}$} 
\end{center}
\end{figure}
 \begin{figure}[h!]
\begin{center}
\includegraphics[scale=0.42]{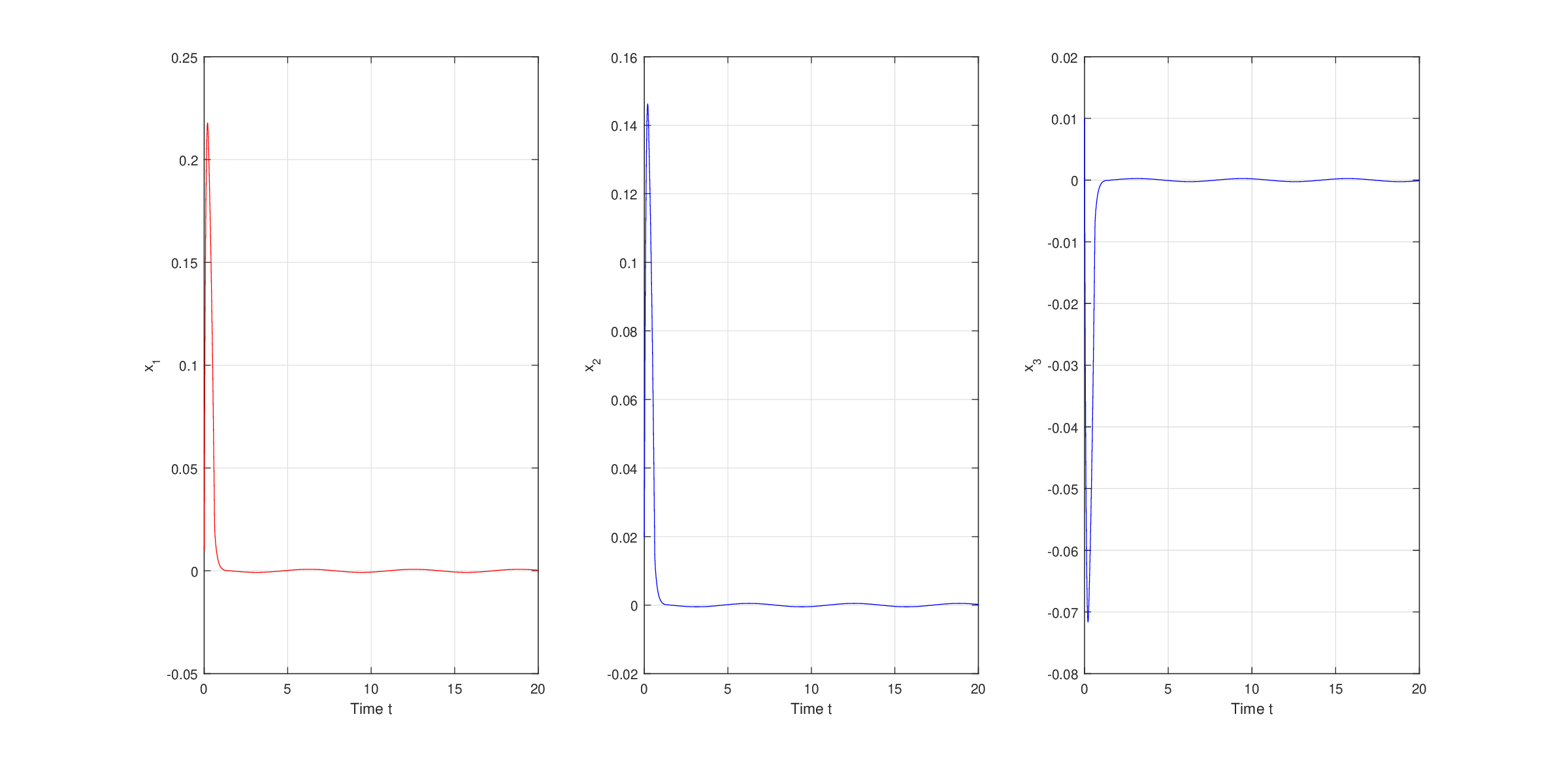}
\caption{The sytem state in $\R^3$ using ${\rm Sign}_\delta(t,x)$ with $\delta(t)=\max\{10^{-3},e^{-5.5t}\}$,  $M=1$, $N=3$} 
\end{center}
\end{figure}
\end{example}

 \begin{example}\normalfont
  Next we consider the system (\ref{sysh}) with 
  $$
  A=\left( \begin{array}{ccc}
-11 &\;\; 5&\;\; 0 \\ \\
9 &\;\; -10 &\;\;  0 \\ \\
0&\;\;0 &\;\;-11
\end{array} \right), B=\left( \begin{array}{ccc}
2  \\ \\
-3  \\ \\
4
\end{array} \right), f_1(x,u)=\left( \begin{array}{ccc}
3u+ 0.8\sin x_2  \\ \\
2u+ 0.9 \cos x_1 \\ \\
-u+0.8\sin x_3
\end{array} \right)
  $$
  
  $$
  f_2(x,u)=\left( \begin{array}{ccc}
3\sin x_2 \\ \\
0 \\ \\
0
\end{array} \right), C=\left( \begin{array}{ccc}
5 &\;\; -3&\;\; 4
\end{array} \right), F = \left( \begin{array}{ccc}
1 &\;\; 0&\;\; 0
\end{array} \right).
  $$
 Suppose that the unknown  $\theta=3 $, the control input $u=8\cos t$ and 
 $$
 \mathcal{F}(x) = \left\{
\begin{array}{lll}
{\rm sign}(x)(2\vert x \vert +5) & \mbox{ if } & x \neq 0,\\
&&\\
\; [ -5,5 ] & \mbox{ if } & x = 0.
\end{array}\right.
$$
The well-posedness of the original system and the sliding mode observer system follows (Theorem \ref{existence}) since we have $B^TQ=C$ where  
  $$
 Q=\left( \begin{array}{ccc}
5/2 &\;\; 0&\;\; 0 \\ \\
0 &\;\; 1 &\;\;  0 \\ \\
0&\;\;0 &\;\;1
\end{array} \right)
$$ 
Then $L_1=0.9, L_2=3, L_3=3$ and $\gamma=9.9$. Then Assumptions 1--4 are satisfied with 
  $$
 P=\left( \begin{array}{ccc}
1 &\;\; 0&\;\; 0 \\ \\
0 &\;\; 1 &\;\;  0 \\ \\
0&\;\;0 &\;\;1
\end{array} \right),\;\;\; L=\left( \begin{array}{ccc}
0  \\ \\
14 \\ \\
0
\end{array} \right), \;\;\;\epsilon = 0.2, \;\;\; K=3.
  $$
 We have a comparison between a standard Luenberger observer and the sliding mode observers (\ref{obs}), (\ref{obsn}) with $x_0=(3\;2\;1)^T, \;\tilde{x}_0=(15\;-20\;11)^T$. Both sliding mode observers has quite same performance  and converge faster than the Luenberger observer obviously (Figure 8).
  \begin{figure}[h!]
\begin{center}
\includegraphics[scale=0.63]{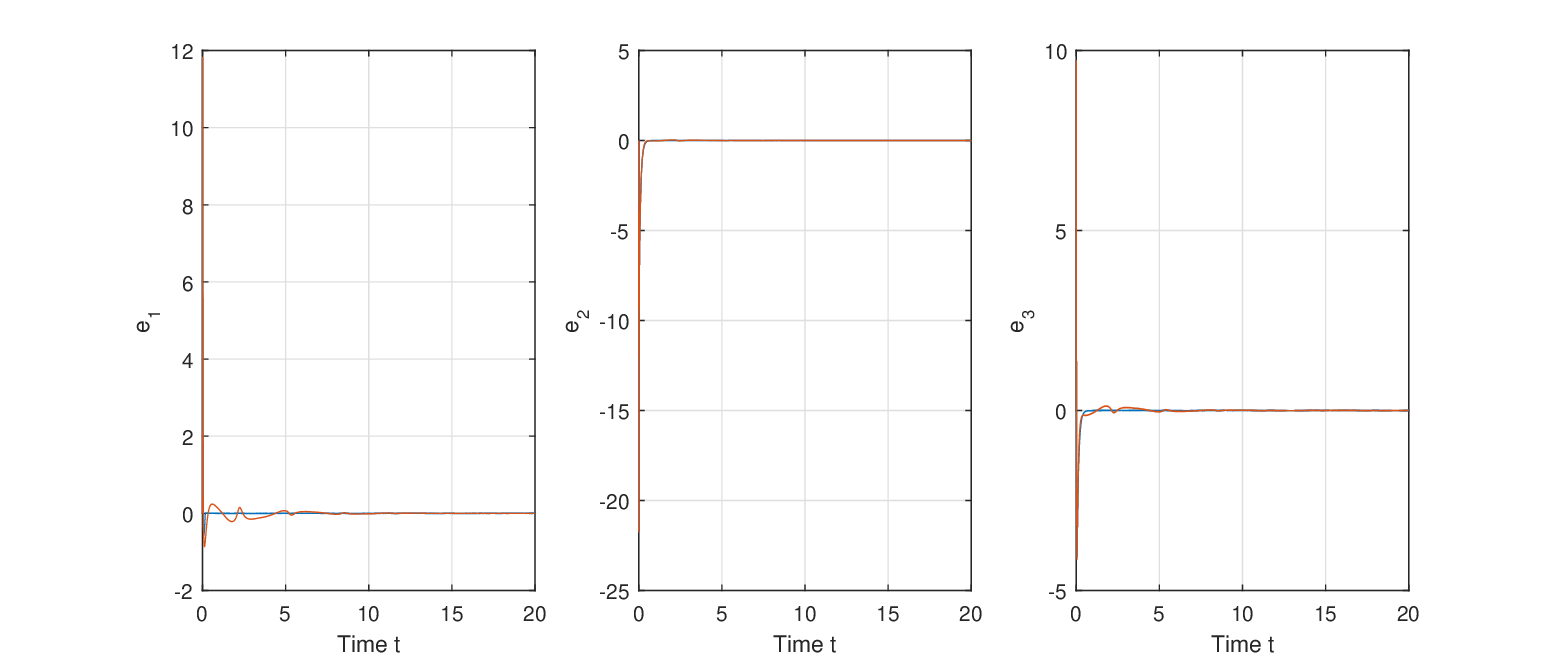}
\caption{Errors using a Luenberger observer (in red) and the sliding mode observer (\ref{obs}) (in blue)} 
\end{center}
\end{figure}
  \end{example}

  \begin{example}\normalfont
 Finally, we present an example to show the applicability of our approach, which
cannot be achieved using the Luenberger observer mentioned in \cite{Huang}. Specifically, we consider the
identical data as in Example 2, with the only difference being that\\
  $$
  A=\left( \begin{array}{ccc}
-1.1 &\;\; 5&\;\; 0 \\ \\
9 &\;\; -1 &\;\;  0 \\ \\
0&\;\;0 &\;\;-1.1
\end{array} \right)$$
and the unknown  $\theta=3\sin t $.
%
%

Let us assign the following values: $L_1=0.9, L_2=3, L_3=3$, $\gamma=9.9$ and $\gamma'=0.9$. It becomes
apparent that, for any matrix  $L$, it is impossible to satisfy the condition $A-LF\le-1.1I $. Basic
calculations show that Assumption 4 cannot be fulfilled, thereby rendering the application of the
Luenberger observer in \cite{Huang} unfeasible. However, by using the same values of  $P, L, \epsilon, K$ as in Example 2, it is straightforward to verify Assumptions 1, 2, 3, 4' of Theorem \ref{tm2}.  
 By using the sliding mode observer (\ref{obsn}) with the gain $\beta =10$, initializing the system with $x_0=(3\;2\;1)^T, \;\tilde{x}_0=(15\;27\;16)^T$ and employing an explicit scheme for $\mathcal{F}$,  Figure 9 shows the convergence of the error $e=\tilde{x}-{x}$ to zero. 
  \begin{figure}[h!]
\begin{center}
\includegraphics[scale=0.68]{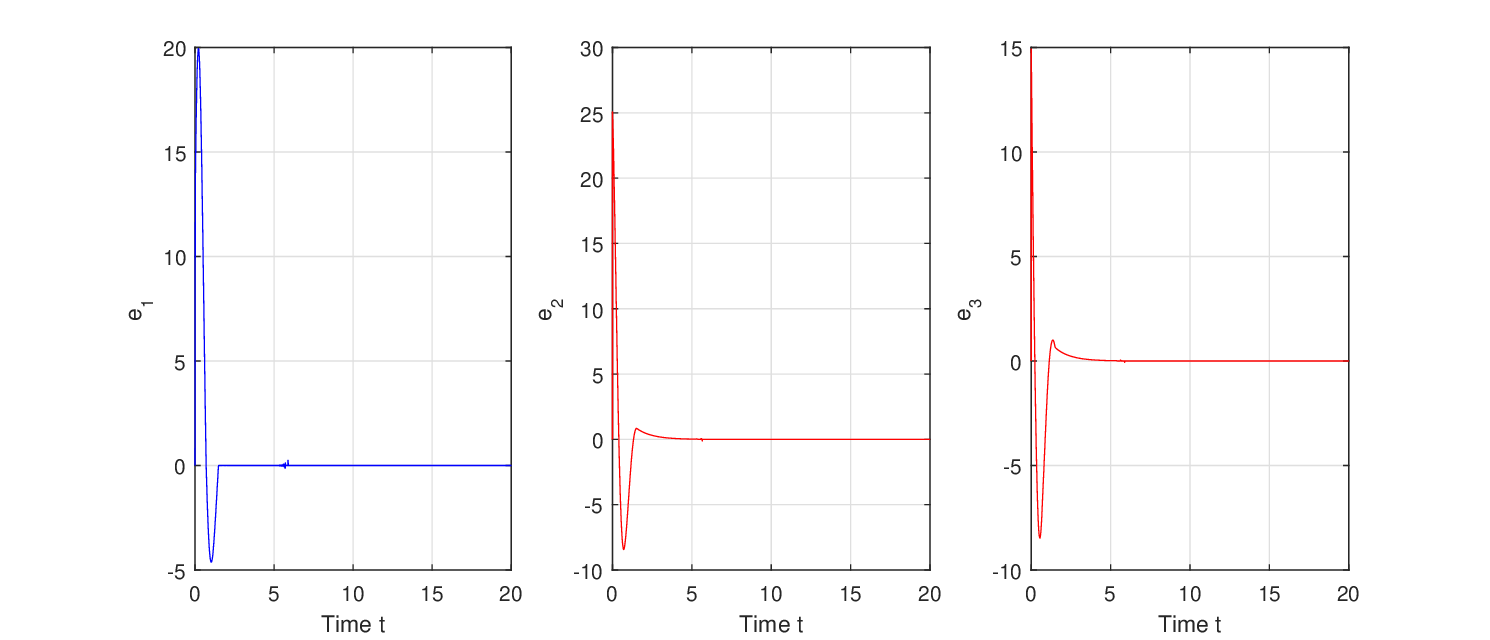}
\caption{The convergence of the sliding mode observer (\ref{obsn})} 
\end{center}
\label{luref}
\end{figure}

 Note that, Theorem \ref{reduce} can be also applied with $K=0, Q=I_2, \epsilon=0.2$.   The reduced-order observer (\ref{obseu}) becomes
\beq \left\{
\begin{array}{lll}
\; \dot{\tilde{z}}&=&\left( \begin{array}{ccc}
-1 &\;\;  0 \\ \\
0 &\;\;-1.1
\end{array} \right) \tilde{z} +\left( \begin{array}{ccc}
-3  \\ \\
4
\end{array} \right)\tilde{\omega}  +\left( \begin{array}{ccc}
9  \\ \\
0
\end{array} \right)x_1\\
 &+&(0 \;\;I_{2})f_1\Big(\left( \begin{array}{ccc}
x_1 \\ \\
\tilde{z} 
\end{array} \right),u\Big)\\
&&\\
\;\tilde{\omega} &\in& -\mathcal{F}(\left( \begin{array}{ccc}
 -3&\;\; 4
\end{array} \right)\tilde{z}+5x_1).
\end{array}\right. 
\eeq
{However, the numerical convergence of the reduced-order observer (\ref{obseu}) fails and {is explosive} in this sensitive case as it relies on an approximation of $x_1$ from the original system. Conversely, the reduced-order observer (\ref{obseu}) can be successfully applied to Example 2, primarily due to the highly negative definiteness of $A$.}
 \end{example} 
\section{Conclusion and perspectives} 
In this paper, the advantages of sliding mode observers are demonstrated for set-valued Lur'e dynamical systems that are faced with uncertainties. The robustness of this observer technique is a significant factor in the analysis and control of such systems. We also  present a new  continuous approximation of the sliding mode technique, which provides improved performance compared to conventional methods. However, {the traditional condition $(\ref{f2p})$} is a limiting factor in the applicability. Further research is needed to explore the possibility of relaxing or removing this condition to increase the range of systems that can be analyzed. This is an area that merits further investigations and has the potential to enhance the performance of sliding mode observers for set-valued Lur'e dynamical systems.\\

\noindent $\mathbf{Acknowledgement.}$
 The authors would like to show their gratitude to the handling Editor and the Reviewers for  their careful reading of our manuscript and for raising many interesting  questions, which helps us to improve the paper significantly.

\end{document}